\documentclass{amsart}
\usepackage{amsmath,amssymb,latexsym,color}
\usepackage{hyperref}
\newtheorem{lemma}{Lemma}

\newtheorem{prop}{Proposition}

\newtheorem{defi}{Definition} 
\newtheorem*{namedthm*}{\namedthmname}
\newenvironment{namedthm}[1]
  {\newcommand\namedthmname{#1}\begin{namedthm*}}
  {\end{namedthm*}}
\newcommand{\PSL}{\mathrm{L}}

\newcommand{\abs}[1]{\left\vert#1\right\vert}         %  valore assoluto
\newcommand{\set}[1]{\left\{#1\right\}}               %  insieme
\newcommand{\seq}[1]{\left<#1\right>}                 %  sottogr generato

\newcommand{\Aut}[1]{\mathrm{Aut}(#1)}

\newcommand{\C}{\mathbb{C}}

\newcommand{\Stab}{\mathrm{Stab}}

\newcommand{\ov}[1]{\overline{#1}}

\begin{document}
 \author[Fumagalli]{Francesco Fumagalli}
    \address{Dipartimento di Matematica  e Informatica ``U.\,Dini" \\
   Universit\`a di Firenze\\
   Viale Morgagni 67A, \ I-50134 Firenze, Italy}
   
   \email{francesco.fumagalli\,@\,unifi.it}
   \author[Leinen]{Felix Leinen}
   \address{Institute of Mathematics \\
   Johannes Gutenberg-University\\
   D$-$55099 \,Mainz, Germany}
   \email{Leinen\,@\,uni-mainz.de}
   
   \author[Puglisi]{Orazio Puglisi}
   \address{Dipartimento di Matematica  e Informatica ``U.\,Dini"\\
   Universit\`a di Firenze\\ 
   Viale Morgagni 67A, \ I-50134 Firenze, Italy}
   \email{orazio.puglisi\,@\,unifi.it}
   
\title{An upper bound for the nonsolvable length of\\ 
a finite group in terms of its shortest law}
% Words, permutations, and the nonsolvable length of a
% finite group (Shalev)
\date{\today}
\dedicatory{In memory of our friend Carlo Casolo}
\begin{abstract}
Every finite group $G$ has a normal series each of 
whose factors is either a solvable group or a direct product of 
non-abelian simple groups. The minimum number of nonsolvable factors, 
attained on all possible such series in $G$, is called the 
\emph{nonsolvable length} $\lambda(G)$ of $G$. 
In the present paper, we prove a theorem about permutation 
representations of 
groups of fixed nonsolvable length. As a consequence, we show that  
in a finite group of nonsolvable length at least $n$, no non-trivial 
word of length at most $n$ (in any number of variables) can be a law. 
This result is then used to give a bound on $\lambda(G)$ in terms 
of the length of the shortest law of $G$, thus confirming a conjecture 
of {\sc Larsen}. Moreover, we give a positive answer to a problem raised by 
{\sc Khukhro} and {\sc Larsen} concerning the non-2-solvable length of 
finite groups.\vskip1ex
\noindent
MSC2020$\colon$ 20F22, 20B05, 20D99.\\ 
Keywords$\colon$ finite groups, nonsolvable length, group laws, word maps.
   
\end{abstract}
\maketitle

%%%%%%%%%%%%%%%%%   Section 1     %%%%%%%%%%%%%%%%%%%%%%%%%%%%%

\pagestyle{myheadings}
\markright{{\large\sc an upper bound for the nonsolvable length}}
\markleft{{\large\sc fumagalli $\mid$ leinen $\mid$ puglisi}}

\section{Introduction}

Let $F_\infty$ be the free group of countably infinite rank, with free generators 
$x_i$ $(i\in\mathbb N$). Consider a word $w=w(x_1, \dots, x_k)\in F_\infty$.
For any group $G$, the map $\phi_w\colon G^k\longrightarrow G$, defined 
by $\phi_w(g_1, \dots, g_k)= w(g_1, \dots, g_k)$, is called the 
\emph{word map} induced by $w$. The \emph{verbal subgroup} $w(G)$ of $G$ 
is generated by the image of $\phi_w$. The word $w$ is said to be a 
\emph{law} in $G$, when $w(G)=1$.  

In recent years, quite a large number of papers was devoted to 
questions related to word maps, leading 
to an impressive series of deep results. 
Particularly remarkable achievements were 
the proof of {\sc Ore}'s conjecture (see \cite{ore}), 
that every element in a finite 
non-abelian simple group is a commutator,
and a series of results on {\sc Waring} type
problems (see \cite{larsen1} and \cite{larsen2}). The 
interested reader should consult the survey \cite{sha} by {\sc Shalev} 
and its extensive 
bibliography in order to get a clear picture of the research in this area.

Given a word $w$ and a group $G$, one may ask under which circumstances 
$w$ is not a law in $G$: Which properties of $G$ ensure, that 
$w(G)\not=1$?   \
One of the first results in this direction is contained in 
\cite{jones}: Given a non-trivial word 
$w$, there exist only a finite number of finite non-abelian simple 
groups, which admit $w$ as a law. In other words, for every non-trivial 
$w$, there exists a natural number $N=N(w)$ such that the
verbal subgroup $w(S)$ of any finite 
non-abelian simple group $S$ of order larger than $N$ is non-trivial. 

Properties of word maps in arbitrary finite groups have been 
investigated in \cite{bors}: In particular,  
for certain words $w=w(x_1,\ldots,x_k)\in F_\infty$, 
the nonsolvable length $\lambda(G)$ of $G$ 
(see above abstract) can be bounded in terms of the sizes of the fibers 
of the verbal map $\phi_w\colon G^k\longrightarrow G$. 
This provides evidence in favor of 
the following conjecture of {\sc Michael Larsen} (see 
\cite[Conjecture 1.3]{bors}).
 
\begin{namedthm}{Conjecture}
There exists a function $g\colon\mathbb N\longrightarrow \mathbb N$ 
such that $\lambda(G)\leq g(\nu(G))$ for every finite group $G$.
Here,\vskip1ex 
\phantom{.}\hfill
$\nu(G)=\min\big\{\left|w\right| ~\,\big|~\, w 
\textrm{ is a non-trival reduced law in } G\big\}$\,,
\hfill\phantom{.}\vskip1ex
\noindent
where $|w|$ denotes the length of the reduced word \,$w\in F_\infty$.
\end{namedthm}

In the present paper, we shall confirm this conjecture in the following form.

\begin{namedthm}{Theorem A}\label{thm:A}
Let $G$ be any finite group. Then $\lambda(G)<\nu(G)$.
\end{namedthm}

Our bound on the non-solvable length can probably be improved:
At least when $w$ is not a commutator word, we can use  
\cite[Proposition 5.10]{fuma} in order to see, that the nonsolvable length 
$\lambda(G)$ is bounded by $\log_2(\left|w\right|)$. 
This argument uses the fact, that a law of the form \,$x_1^n\,\tilde w$ 
\,with $\tilde w\in F_\infty'$ always implies the law  $x_1^n$.

This observation encourages us to set up the following conjecture.

\begin{namedthm}{Conjecture}
The nonsolvable length of any finite group $G$ can be bounded by a function
in $O(\log(\nu(G)))$.
\end{namedthm}

We shall obtain Theorem A as a consequence of the following main result.

\begin{namedthm}{Theorem B}\label{thm:B}
Let $G$ be a finite group of nonsolvable length $\lambda(G)=n$. 
Then, for every non-trivial word $w=w(x_1,\ldots,x_k)\in F_\infty$ 
of length at most $n$, 
there exists $\ov{g}=(g_1, \dots, g_k)\in G^k$ such that 
$w(\ov{g})=w(g_1, \dots, g_k)\not=1$. 
\end{namedthm}

When proving Theorem B, one realizes immediately, that a detailed knowledge 
of the action of the group on the components of certain non-abelian 
composition factors is needed. For this reason, we shall derive Theorem B 
from a result 
about permutation representations of groups of fixed nonsolvable length,
which might be of interest in its own right. 

In order to formulate this result, we introduce some further
notation. Let $w=y_1y_2\cdots y_n$ \big(with 
$y_i\in\{x_1^{\pm1},\ldots,x_k^{\pm1}\}$\big)
be a reduced word in $F_\infty$.
Then we shall need to consider its \emph{partial subwords}, namely the 
words  
$w_i = y_1y_2 \ldots y_i$ for $0\le i\le n$ (where $w_0=1$). 

\begin{defi}\label{def:P_n} 
Let $G$ be a finite group acting faithfully on the set $\Omega$.\vskip0ex
We say that $G$ satisfies property $\mathcal P_n$ 
in its action on $\Omega$ \big(and write $G\in\mathcal P_n(\Omega)$\big), if 
for every non-trival reduced word 
$w=w(x_1,\ldots,x_k)\in F_\infty$ of length $n$, 
there exist $\omega\in \Omega$ and $\overline{g}\in G^k$, such that 
the sequence $\{ \omega w_i(\overline{g})\}_{i=0}^{n}$
consists of $n+1$ distinct elements in $\Omega$.\vskip0ex
We write $G\in\mathcal P_n$ whenever $G\in\mathcal P_n(\Omega)$ 
for every faithful $G$-set $\Omega$.
\end{defi}

Clearly, if $G\in\mathcal P_n(\Omega)$ for some faithful $G$-set 
$\Omega$, then $w(G)\neq 1$ for every non-trivial reduced word $w$ 
of length at most $n$. 
However, this property is much stronger than the non-triviality of $w(G)$. 
Indeed, when $G\in\mathcal P_n(\Omega)$, it is possible to find 
some $\ov{g}\in G^k$ such that all the elements  
$w_1(\ov{g}), w_2(\ov{g}), \dots , w_n(\ov{g})=w(\ov{g})$ are non-trivial 
and pairwise distinct.

Theorem B will be a consequence of the following result.
\begin{namedthm}{Theorem C}\label{thm:C}
Let $G$ be a finite group with $\lambda(G)=n$, and let $\Omega$ be a
faithful transitive $G$-set. Then, for every $\omega\in\Omega$ and for every 
non-trivial reduced word 
$w=w(x_1,\ldots,x_k)\in F_\infty$ of length $n$,
there exist a Sylow $2$-subgroup $P$ of $G$ and a tuple $\overline{g}\in P^k$ 
such that the points \
$
\omega w_0(\overline{g}),\,\omega w_1(\overline{g}),\, 
\dots ,\, \omega w_n(\overline{g})
$ \
are pairwise distinct. In particular, $G\in\mathcal P_n$. 
\end{namedthm}

Theorem C says: If $G$ has nonsolvable length $n$, 
and if $H$ is a core-free subgroup of $G$, then for every 
non-trivial word $w=w(x_1,\ldots,x_k)$ of length $n$ there exists a $k$-tuple
$\ov{g}$ of elements from a Sylow $2$-subgroup of $G$ such that the cosets 
$$ H,\ Hw_1(\ov{g}),\ \dots\ldots,\ H w_{n-1}(\ov{g}),\  H w_{n}(\ov{g})=H w(\ov{g})$$ 
are pairwise distinct.

The proof of Theorem C relies heavily on the existence of so called 
\emph{rarefied} subgroups. 
In \cite[Theorem 1.1]{fuma} it has been proved, that
every finite group of nonsolvable length $m$ contains 
\emph{$m$-rarefied} subgroups. These are subgroups  
of the same nonsolvable length $m$, with a very restricted structure. 
In particular, the non-abelian composition factors of an $m$-rarefied group 
belong to the class
$$\mathcal{L}=\big\{\PSL_2(2^r),\, \PSL_2(3^r),\, \PSL_2(p^{2^a}),\, 
\PSL_3(3),\, \null^2B_2(2^r)~\,\big\vert~\,
 p,r\textrm{ odd primes}, \, 
 a\in \mathbb N\big\}.$$
The precise definition of $m$-rarefied 
groups will be given in section \ref{sec:strategy}.

In proving Theorem C, we shall reduce ourselves to the case of rarefied groups 
and then use the constraints on the structure of such 
groups in order to complete the proof. 
The proof of Theorem C depends on the classification of finite simple groups. 

We finally note, that Theorem C partly solves a question of {\sc Evgenii Khukhro} 
and {\sc Pavel Shumyatsky}. They have shown in \cite[Theorem 1.1]{khukhro}, that
the nonsolvable length of any finite group $G$ is bounded by $2L_2(G)+1$, where $L_2(G)$
denotes the maximum of the 2-lengths of the solvable subgroups of $G$.
This bound was improved in \cite[Theorem 5.2]{fuma} to $L_2(G)$.
In this context, {\sc Khukhro} and {\sc Shumyatsky} raised the following question 
\cite[Problem 1.3]{khukhro}.

\begin{namedthm}{Problem}
For a given prime $p$ and a given proper group varity $\mathfrak V$, is there a bound
for the non-$p$-solvable length of finite groups whose Sylow $p$-subgroups belong 
to $\mathfrak V$?
\end{namedthm}

Theorem C gives an affirmative answer in the case when $p=2$.

%%%%%%%%%%%%%%%% Section 2 %%%%%%%%%%%%%%%%%%%%%%

\section{Preliminaries on group actions and words}
Consider any non-trivial reduced word $w=w(x_1,\ldots,x_k)\in F_\infty$
of length $n$. Let $w=y_1\cdots y_n$ with $y_j=x_{i_j}^{\epsilon_j}$ and
$\epsilon_j\in\{\pm1\}$ for all $j$. As before, we let $w_0=1$ and
$$
w_j=y_0\cdots y_j\quad\text{and}\quad 
\widetilde w_j=\left\{
\begin{array}{ll}
w_j^{-1} & \text{if}~~\epsilon_j=-1\\[.5ex]
w_{j-1}^{-1} & \text{if}~~\epsilon_j=+1
\end{array}
\quad\text{for}\quad 1\le j\le n.
\right.
$$

For $k$-tuples $\ov g=(g_1,\ldots,g_k)$ and 
$\ov h=(h_1,\ldots,h_k)$ of elements in a group $G$, we let
$\ov g\ov h=(g_1h_1,\ldots,g_kh_k)$.\vskip1ex 

%Whenever $A$ is a
%subgroup of $G$, we write $\ov{g}\equiv \ov{h} \pmod{A}$ instead of $Ag_i=Ah_i$, 
%for $1\le i\le k$. \\

The proof of the following lemma is straightforward 
(see \cite[Lemma 2.1]{leinen}).

\begin{lemma}\label{lem:word_basic_0}
Let $N\trianglelefteq G$. 
Consider $\ov{g}\in G^k$ and $\ov{b}\in N^k$. Then,\vspace*{1ex} 
in the above notation,\quad
$
w(\ov{b}\ov{g})=a_1^{\epsilon_1}\cdots 
a_n^{\epsilon_n}w(\overline{g})\in Nw(\ov{g})\quad\text{with}\quad
a_j=(b_{i_j})^{\widetilde{w}_{j}(\overline{g})}\quad\text{for all }~j.
$\vskip2ex
\end{lemma}

The following situation will come up several 
times in the course of our analysis: 
A normal subgroup of a finite group $G$ is 
a direct product of 
subgroups $M_i$ $(0\le i\le m)$, which are  
all isomorphic to a certain group $M$ and which are 
permuted transitively under the conjugation with elements from $G$. 

\begin{lemma}\label{induced module}
In the above situation, let $H=N_G(M_0)$ 
and $\widehat{M_{0}}=\prod_{i\not=0}M_i$. 
Suppose, that 
$\mathbb F$ is a field and $W$ is an irreducible 
$\mathbb F H$-module satisfying $[W,M_0]\neq 0$ and 
$[W,\widehat{M_{0}}]=0$.
Then the induced $\mathbb FG$-module 
$V=\mathrm{Ind}{\stackrel{G}{_H}}(W)$ is irreducible.
\end{lemma}
\begin{proof}
Let $T$ be  a transversal to $H$ in $G$ 
containing $1$. 
Then
$$ V=\mathrm{Ind}{\stackrel{G}{_H}}(W)=
W\otimes_H\mathbb FG=
\oplus_{t\in T}(W\otimes_{H} t) $$
and each of the subspaces $W_t=W\otimes_H t$ $(t\in T)$ is a $G$-block. 
If  
$i\in\set{0,1,\ldots,m}$  and $t\in T$, then 
$[W_t,M_i]=0$ 
whenever $M_i\not={M_0}^t$. The subgroup $M_0$ has 
trivial centralizer in $W_1$, 
because $W$ is an irreducible $H$-module, $M_0$ is 
normal in $H$ and acts non-trivially on $W$. 
If $U$ is a $G$-submodule of $V$ and $u\in U$ is a  
non-trivial element, 
write $u=\sum_{t\in T}u_t$, with $u_t\in W_t$ for  
all $t\in T$. 
Without loss of generality, we may assume 
that $u_1\not=0$ and pick $x\in M_0$ not 
centralizing $u_1$.
Then $[u,x]=u_1(x-1)$ is still in $U$, hence 
$U\cap W_1\not=0$. 
Since $U\cap W_1$ is an $H$-module, it must be 
the whole $W_1$. 
Once we know that $W_1\leq U$ we immediately 
deduce $U=V$.
\end{proof}

The following lemma generalizes an idea from 
\cite{batta}.

\begin{lemma}\label{wreath property}
Let $G,H, W$ and $V$ be as in Lemma 
\ref{induced module} and adopt the notation introduced 
in its proof. In particular, $T$ denotes a transversal to $H$ in $G$ with $1\in T$.
Suppose, that the subgroup 
$A_0$ of $M_0$ satisfies $[W,A_0]\neq 0$.
Let 
$$M^*=\prod_{t\in T}M_0^t\qquad\text{and}\qquad
A^*=\prod_{t\in T}A_0^t~.$$
Suppose further, that $w=y_1\cdots y_n$ is a non-trivial reduced 
$k$-variable word of length 
$n$, and that there exists some $\ov{g}\in (N_G(A^*))^k$  
such that the groups 
$M_0^{w_{j}(\overline{g})}$ $(0\le j\le {n-1})$ 
are pairwise distinct factors of the above direct product $M^*$.

Then, for every non-trivial $z\in W$, there exist 
$u\in\langle zH\rangle$ and 
$\ov{b}\in (A^*)^k$  such that the $n+1$ elements \vspace*{-2ex}
$$(u\otimes 1) {w_{j}({\ov b}{\ov g})}\qquad (0\le j\le n)$$
are pairwise distinct elements of $V=W\otimes_H\mathbb{C}G$.
Here, the number of choices for $\ov{b}$ is at least 
$\left|A^*\right|^k/\left|A_0\right|$.
\end{lemma}

\begin{proof}
For convenience of notation we let $T=\{t_0,\ldots,t_m\}$ with $t_0=1$ and
$M_i=M_0^{t_i}$ and $A_i=A_0^{t_i}$ for all $i$.
Without loss we may assume that
$A_0^{w_j(\ov{g})}=A_j$ for $0\le j\le n-1$.
Clearly, $A_j$ acts non-trivially on the block
$W_i = W\otimes t_i$ if and only if $i=j$. 

Suppose firstly, that $A_0^{w_n(\ov g)}\neq A_j$ for $0\le j\le n-1$.
Then, for every $u\in W_1$, the vectors $(u\otimes1)w_j(\ov g)$ $(0\le j\le n)$
belong to pairwise different blocks and are, therefore, all distinct. 
In this situation any $k$-tuple $\ov{b}\in (A^*)^k$ gives rise to a sequence 
$\set{(u\otimes 1) {w_{j}({\ov b}{\ov g})}}_{j=0}^{n}$
with $n+1$ distinct elements.  Since the number of such tuples
$\ov{b}$ is $\left|{A^*}\right|^k$, the claim holds.

Suppose next, that there exists some $\ell\in\{0,\ldots n-1\}$ such that 
$$A_0^{  w_{n}(\overline{g})}=A_0^{  w_{\ell}(\overline{g})}=A_\ell\,.$$
\noindent
Modulo application of a suitable Nielsen transformation to the 
variables $x_1,\ldots,x_k$ we may assume without loss, that $y_n=x_k$  
is the last letter occurring in the word $w$.
We focus our attention on the word 
$\mu=w_\ell^{-1}w=y_{\ell+1}\cdots y_n$ 
and on its partial subwords 
$\mu_j = w_\ell^{-1}w_{\ell+j}=y_{\ell+1}\cdots y_{\ell+j}$ 
for $0\le j\le n-\ell$.

For every $j\in\{0,\ldots,n-\ell\}$ and every $\ov{b}\in {(A^*)}^k$, we
define elements $f_j(\ov{b})\in M^*$ by the equation \vspace*{-2ex}
$$ f_j(\ov{b})\mu_j(\ov{g})=\mu_j(\ov{b}\ov{g}).$$
Each $\mu_j(\ov{g})$ is in $N_G(A^*)$ and, using  Lemma 
\ref{lem:word_basic_0}, is  readily seen that each $f_j(\ov{b})$ belongs to $A^*$.
Fix any non-trivial vector $z\in W$ and consider  $z\otimes t_\ell\in W_\ell=W_1w_\ell(\ov{g})$.
Choose now $b_1, \dots, b_{k-1}$ freely in $A^*$. This choice can be done 
in $\left|A^*\right|^{k-1}$ different ways. The choice of $b_k$ will be 
described just below.  

In order to express some relations in a
clearer fashion, we will interpret elements $d\in M^*$ as functions 
$d\colon\{0,\ldots,m\}\longrightarrow M$ with values $d[i]$ $(0\le i\le m)$.
Notice that $W_\ell\mu(\ov{g})=W_\ell$. We need to choose $b_k\in A^*$ in such a way that 
$(z\otimes t_\ell)\mu(\ov{b}\ov{g})\neq z\otimes t_\ell$. 
To this end,  we choose $b_k$ under the only constraint that 
\begin{equation}\label{eq:mu_0}
(z\otimes t_\ell)\mu(\ov b\ov g) = 
(z\otimes t_\ell)f_{n-\ell-1}(\ov{b})\mu_{n-\ell-1}(\ov{g})(b_kg_k)
\not=z\otimes t_\ell.
\end{equation}
\noindent
Note that 
\begin{align*}
(z\otimes t_\ell)f_{n-\ell-1}(\ov{b})\mu_{n-\ell-1}(\ov{g})(b_kg_k) &= 
(z\otimes t_\ell)f_{n-\ell-1}(\ov{b})b_k^{\mu_{n-\ell-1}(\ov{g})^{-1}}\mu(\ov{g})\\
&=z\Big(f_{n-\ell-1}(\ov{b})b_k^{\mu_{n-\ell-1}(\ov{g})^{-1}}\Big)[\ell]\otimes t_\ell\mu(\ov{g})\\
&=v\otimes t_\ell
\end{align*}
where $v=zf_{n-\ell-1}(\ov{b})[\ell](b_k[\ell.\mu_{n-\ell-1}(\ov{g})])^{\tau}h$, 
for some automorphism $\tau$ of $A_0$ and for some 
$h\in H$ defined by $t_l\mu(\ov{g})=ht_l$.
Inequality (\ref{eq:mu_0}) holds if and only if 
\begin{equation}\label{eq:mu_1}
f_{n-\ell-1}(\ov{b})[\ell](b_k[\ell.\mu_{n-\ell-1}(\ov{g})])^{\tau}h\not\in \Stab_H(z).
\end{equation}
The calculation of $f_{n-\ell-1}(\ov{b})[\ell]$ does not need the knowledge of 
the element\linebreak
$b_k[\ell.\mu_{n-\ell-1}(\ov{g})]$, as can be checked easily 
from the identities of Lemma \ref{lem:word_basic_0} and the fact that, 
for all $0\leq i<j< n-l$ the subgroups ${A_l}^{\mu_i(\ov{g})}$ and 
${A_l}^{\mu_j(\ov{g})}$ are distinct. 
We let $s=f_{n-\ell-1}(\ov{b})[\ell]$ and show that there exists some 
element $c\in A_0$ such that 
\begin{equation}\label{eq:mu_2}
sc^{\tau}h\not\in \Stab_H(z).
\end{equation}
If $c$ has been found, then the element $b_k$ can be defined by setting 
$b_k[\ell.\mu_{n-\ell-1}(\ov{g})]=c$, while $b_k[i]$ can be chosen freely
for all $i\neq \ell.\mu_{n-\ell-1}(\ov{g})$.

If $z_1=z s\not=z h^{-1}=z_2$, then the coset
$\{y\in H\mid z_1y=z_2\}$ of $\mathrm{Stab}_H(z_1)$ is
distinct from $\mathrm{Stab}_H(z_1)$. Therefore, this coset cannot 
contain any subgroup. In particular, it cannot contain $A_0^\tau$.
In this situation, we can select $c\in A_0$ such, that the inequality  
(\ref{eq:mu_2}) holds, and the choice of $b_k$ is completed.

We are left with the situation when $z s=zh^{-1}$. 
Note that $z$ was an arbitrary non-trivial element from $W$ up to now.
Pick $r\in H$ and consider $u=zr$. If $us\not= uh^{-1}$ 
we replace $z$ by $u$ and apply the above argument with $u$ instead of $z$. 
Thus, it remains to study the case when $(z r)s= (z r)h^{-1}$ for all 
$r\in H$. 
Since $A_0$ acts non-trivially on $W=\langle zH\rangle$, there exist certain $r\in H$ and
$c\in A_0$ such that $ (z r)s c^\tau\not=(z r)s=(z r)h^{-1}$ and the choice of $b_k$ can again be completed using $u=zr$ in place of $z$.

Using the chosen element $b_k$, we now have
\begin{equation*}
(zr\otimes t_\ell)\mu(\ov{b}\ov{g})=zrsc^{\tau}h\otimes t_\ell\neq 
zr\otimes t_\ell
\end{equation*}
for suitable $r\in H$, and the elements
$(u\otimes 1)w_0(\ov{b}\ov{g}), \dots, (u\otimes 1)w_n(\ov{b}\ov{g})$
are all distinct for $u=zr$. 
Because the element $b_k$ could be chosen freely in all but one of its components, 
the possible choices for $b_k$ are at least 
$\left|A^*\right|/\left|A_0\right|$, hence the possible choices for the 
$k$-tuples $\ov b$ satisfying the requirements of the Lemma 
are at least $\left|A^*\right|^k/\left|A_0\right|$.
\end{proof}

We will also need the following fact.

\begin{lemma}\label{l:module} 
Let $G\leq \mathrm{Sym}(\Omega)$ be a finite transitive group with 
a non-trivial normal subgroup $M^*=\prod_{i=0}^m M_i$, where the subgroups 
$M_i$ are permuted transitively by $G$ under the action by conjugation. 
Suppose that $\omega\in\Omega$, and
let $R$ be a subgroup of $H=N_G(M_0)$ containing $(G_\omega\cap H)\widehat{M_0}$,
where $\widehat{M_0}=\prod_{i\not=0}M_i$. 

If $W$ is an irreducible constituent of
$\mathrm{Ind}{\stackrel{H}{_R}}(\mathbb C)$, on which $M_0$ acts 
non-trivially, then $\mathbb{C}\Omega$ has a submodule isomorphic 
to $V=\mathrm{Ind}{\stackrel{G}{_H}}(W)$.
\end{lemma}

\begin{proof}
The module 
$V=\mathrm{Ind}{\stackrel{G}{_H}}(W)$ is irreducible
by Lemma \ref{induced module}.
Therefore it suffices to show, that
$\mathrm{Hom}_G(\mathbb{C}\Omega, V)\not=0$. 
Frobenius reciprocity(\cite[Theorem 10.8]{CurtisReiner}) gives
$$
\mathrm{Hom}_G(\mathbb{C}\Omega, V)\ \simeq\
\mathrm{Hom}_{G_\omega}(\mathbb{C}, V)
\ \simeq\ C_V(G_\omega)\ \simeq\ \mathrm{Hom}_{G_\omega}(V,\mathbb{C}).
$$
In particular, we only need to show that $\mathrm{Hom}_{G_\omega}(V,\C)\not=0$.

The $G$-module $V$ is itself an induced module. Therefore, by 
application of Mackey's theorem (\cite[Theorem 10.13]{CurtisReiner}), 
the $G_\omega$-module $V$ contains   
$\mathrm{Ind}{\stackrel{G_\omega}{_{H_\omega}}}(W)$ as a direct summand. 
Hence it suffices to show, that 
$\mathrm{Hom}_{G_\omega}(\mathrm{Ind}{\stackrel{G_\omega}{_{H_\omega}}}(W),\C)\neq0$.

Note that $H_\omega\leq R$, and that the $H$-module $W$ is a submodule of 
$\mathrm{Ind}{\stackrel{H}{_R}}(\mathbb C)$ by hypothesis.
Together with two further applications of Frobenius reciprocity, this implies
\begin{eqnarray*}
\mathrm{Hom}_{G_\omega}(\mathrm{Ind}{\stackrel{G_\omega}{_{H_\omega}}}(W),\C)
&\!\simeq\!& \mathrm{Hom}_{H_\omega}(W, \mathbb C)\\ 
&\!\ge\!&
\mathrm{Hom}_{R}(W, \mathbb C)\ \simeq\
\mathrm{Hom}_{H}(W,\mathrm{Ind}{\stackrel{H}{_R}}(\mathbb C))\ \neq\ 0.
\end{eqnarray*}
\end{proof}

\begin{lemma}\label{projec}
Let $G\leq \mathrm{Sym}(\Omega)$ be a finite transitive group and 
$\omega$ any point of $\Omega$. If $V$ is a non-trivial $G$-submodule of 
$\mathbb{C}\Omega$ and $\pi$ is the projection from $\mathbb{C}\Omega$ onto $V$, 
then $(\omega)\pi\neq 0$.
\end{lemma}
\begin{proof}
Write $\mathbb{C}\Omega=V\oplus U$ with $U$ $G$-invariant.
If $(\omega)\pi = 0$ then $\omega \in U$ and therefore $\omega g\in U$ for every $g\in G$.
Hence $\Omega\subseteq U$ because $G$ is transitive on $\Omega$. It follows that 
$U=\mathbb{C}\Omega$ and  $V=0$, a contradiction. 
\end{proof}

%%%%%%%%%%%%%%%% Section 3 %%%%%%%%%%%%%%%%%%%%%%

\section{$m$-rarefied groups and strategy of the proof}
\label{sec:strategy}
For every finite group $G$ we define  
the following characteristic subgroups:
\begin{align*}
R(G)&=\seq{A\mid A \textrm{ is a  normal solvable  subgroup of } G},\\
S(G)&=\seq{B\mid B \textrm{ is a minimal normal non-abelian 
subgroup of } G}.
\end{align*}
The \emph{$RS$-series} of $G$ is then defined recursively by
$$\begin{cases} R_1(G)=R(G) & \\
    S_i(G)/R_i(G)=S(G/R_i(G)), & \textrm{for } i\geq 1\\     
    R_{i+1}(G)/S_i(G)= R(G/S_i(G)), & \textrm{for } i\geq 1.
   \end{cases} $$
The series thus defined always reaches $G$. If $G=R_{m+1}(G)>R_m(G)$, 
then $m$ is said to be \emph{the nonsolvable length} of $G$ and we write 
$\lambda(G)=m$.\vskip1ex 

By \cite[Theorem 1.1]{fuma}, every group $G$ of nonsolvable length $m$ has 
an $m$-\emph{rarefied} subgroup, 
that is, a subgroup $H$ of nonsolvable length $m$ 
satisfying the following conditions:
\begin{enumerate}
\item[(i)~] 
$R_1(H)=\Phi(H)$ and $R_{i+1}(H)/S_i(H)=\Phi(H/S_i(H))$ for all 
$i=1, \dots, m-1$;\vspace*{.5ex}
\item[(ii)~] 
$S_i(H)/R_i(H)$ is the unique  minimal normal subgroup 
of $H/R_i(H)$  
for\\ $1\le i \le m$;\vspace*{.5ex}
\item[(iii)~] 
the simple components of $S_i(H)/R_i(H)$ are isomorphic to groups 
in the set\\[.5ex] 
$\mathcal{L}\,=~\big\{\PSL_2(2^r),\, \PSL_2(3^r),\, \PSL_2(p^{2^a}),\, 
\PSL_3(3),\, \null^2B_2(2^r)~\,\big\vert~\,
 p,r\textrm{ odd primes}, \, 
 a\in \mathbb N\big\}
 $\\[.5ex]
for $1\le i \le m$.\vspace*{1ex}
\end{enumerate}

\begin{lemma}\label{dihedral}
Let $S$ be a simple group in $\mathcal L$. 
Then $S$ contains an $\Aut{S}$-invariant  
$S$-conjugacy class of subgroups, which are dihedral of 
order $2p$ for some odd prime $p$. 
\end{lemma}

\begin{proof}
We first claim that $S$ contains a conjugacy class of 
(maximal) subgroups that are dihedral of 
order $2m$, for some odd integer $m$.

When the characteristic of $S$ is $2$, choose the 
conjugacy class of normalizers of  
split-torii subgroups (diagonal subgroups) 
of order $q-1$. Each of these subgroups is
a dihedral group of order $2(q-1)$ 
(see \cite[Theorems 6.5.1 and 6.5.4]{GLS3}).

When $S\simeq \PSL_2(q)$,  with $q=3^r$ or 
$q=p^{2^a}$ ($p,r$ odd primes and $a\geq 0$), 
then $S$ contains two conjugacy classes of 
maximal subgroups, which are dihedral subgroups 
of orders respectively $q-1$ and $q+1$. 
Each of these subgroups is the normalizer 
respectively of a split torus (diagonal 
subgroup of $S$) or of a non-split torus.  
According to the congruence class of 
$q$ modulo $4$, take the class of subgroups 
having order $2m$, with $m$ odd 
(see \cite[Theorem 6.5.1]{GLS3}).

When $S=L_3(3)$ the normalizers of the Sylow 
$13$-subgroups form a conjugacy class of 
dihedral groups of order $26$ 
(see \cite{atlas}).

Finally, consider one of the dihedral groups $E$ in the conjugacy class just established.
If $|E|=2m$, then choose a prime $p$ dividing $m$. 
The group $E$ has a unique 
subgroup $C$ of order $p$. And if $a$ is any involution in $E$, then 
$D=C\langle a\rangle$ is a dihedral group of order $2p$. 
Since all the involutions of $E$ are conjugate in $E$, any 
other subgroup of $E$ of order $2p$ is conjugate to $D$.
The set $\{D^g\mid g\in S\}$ is then a family of subgroups 
with the required properties.
\end{proof}

Theorem C will be proved by contradiction. 
So assume, that there exists a group $G$ with nonsolvable length $n$,
such that the conclusion of Theorem C is not valid 
with respect to the action of $G$ on a certain faithful transitive $G$-set $\Omega$. 
Among all possible such counterexamples, we select a pair  $(G,\Omega)$ 
such that $n+\abs{G}+\abs{\Omega}$ is minimal.  

Consider an $n$-rarefied subgroup $H$ of $G$. 
From \cite[Proposition 4.7]{fuma}, we obtain $\lambda (H/C_H(\Gamma))=n$ 
for at least one $H$-orbit $\Gamma$ in $\Omega$.
Hence, by minimal choice of $(G,\Omega)$,  we have  
$G=H$, that is, $G$ is an $n$-rarefied group. 
By \cite[Proposition 4.2]{fuma}, 
every proper subgroup $U$ of $G$ satisfies $\lambda(U)\le n$.
Hence minimality of $(G,\Omega)$ yields, that
every proper subgroup of $G$ has nonsolvable length strictly 
smaller than $n$.\vskip1ex

%%%%%%%%%%%%%%%% Section 4 %%%%%%%%%%%%%%%%%%%%%%

\section{Proof of Theorem C: reduction to the case $\Phi(G)=1$.}
\label{sec:Frattini}
{\bf In this section, $\mathbf{(G,\Omega)}$ is always a minimal 
counterexample to
Theorem C, as above.}
$F=\Phi(G)$ denotes the Frattini subgroup of $G$, and $L$ denotes 
the last term of the derived series of $S_1(G)$. 
Since $G$ is $n$-rarefied, $S_1(G)/F$ is the unique minimal normal
subgroup in $G/F$. In particular, $S_1(G)=FL$. 
\begin{lemma}\label{lem_0}
Assume that $K$ is a normal subgroup of $G$ not containing $L$. 
Then $K\leq F$ and $G/K$ is $n$-rarefied.
\end{lemma}

\begin{proof} 
Suppose that $K\cap L$ is not contained in $F$.
Then $(K\cap L)F/F$ is a non-trivial normal subgroup of $G/F$ 
and therefore it contains $S_1(G)/F$, the unique minimal normal 
subgroup of $G/F$. Thus
$S_1(G)=(K\cap L)F$ and 
$$L=L\cap S_1(G)=L\cap (K\cap L)F=(L\cap F)(K\cap L).$$
Since $L$ is perfect, we get $L=L'\leq(L\cap F)'(K\cap L)$. 
Because $L\cap F$ is nilpotent, iteration of this argument finally yields 
$L=K\cap L$, that is, $L\leq K$. But this contradicts the hypothesis of 
the Lemma.
Hence $K\cap L\leq F$, and this implies that $KF/F$ centralizes 
$LF/F=S_1(G)/F$. 
On the other hand $S_1(G)/F$ has trivial centralizer in $G/F$ 
(by \cite[Lemma 2.4]{fuma}). It follows that $K\leq F$. Hence $G/K$ 
is $n$-rarefied by 
\cite[Proposition 4.2 and Lemma 4.3]{fuma}. 
\end{proof} 

\begin{lemma}\label{lem_1}
If $G_{\omega}<H\leq G$, for some $\omega\in \Omega$, then $L\leq H_G$. 
Moreover $L$ lies in the normalizer of any block system 
(with respect to the action of $G$ on $\Omega$).
\end{lemma}

\begin{proof}
Let $\mathcal{H}=\{Hg\mid g\in G\}$. The translation action of $G$ on 
$\mathcal{H}$ has kernel $H_G$, and 
$\left|\mathcal{H}\right|<\left|\Omega\right|$. 
Assume, that $H_G$ does not contain 
$L$. Then, by Lemma \ref{lem_0}, $G/H_G$ is $n$-rarefied. And since 
$\left|G/H_G\right|+\left|\mathcal{H}\right|<
\left|G\right|+\left|\Omega\right|$, 
the group $G/H_G$ satisfies Theorem C in its action on $\mathcal{H}$. 
For each $k$-variable word $w$  of length $n$ and for every $\alpha\in 
\mathcal{H}$, 
we can thus find a Sylow 2-subgroup $Q/H_G$ and 
$\ov{y} \in (Q/H_G)^k$ 
such that $\big|\{\alpha w_i(\ov{y})\mid 0\le i\le n\}\big|=n+1$. 
Choose $\alpha=H$. 
Note that $Q/H_G=PH_G/H_G$ for some Sylow 2-subgroup $P$ in $G$. 
Let $\ov{x}\in P^k$ be a preimage to $\ov y$. 
Since the points $\alpha w_i(\ov{y})$ $(0\le i\le n)$ are pairwise 
distinct, the products $w_i(\ov{x}) w_j(\ov{x})^{-1}$ $(0\le i<j\le n)$ 
do not lie in $H$ and not in $G_\omega$. 
But then the points $\omega w_i(\ov{x})$ $(0\le i\le n)$
are pairwise distinct too, showing that $G$ satisfies Theorem C, 
a contradiction. 

Let $K$ be the normalizer in $G$ of the block system $\mathcal{B}$.
Note that $KG_\omega$ is a subgroup. 
Assume, that $K=1$. So $G$ acts 
faithfully on $\mathcal{B}$ and, since 
$\left|\mathcal{B}\right|<\left|\Omega\right|$,
we have that $G$ satisfies the conclusion of Theorem C in its action on 
$\mathcal{B}$. It follows, that 
$G$ satisfies the conclusion of Theorem C in its action on $\Omega$, 
a contradiction.
Now $K\neq 1$, whence $KG_\omega>G_\omega$.
We conclude, that $L$ is contained in the 
core of $KG_\omega$, which equals $K$.
\end{proof}

\begin{prop}\label{frattini} The Frattini subgroup $F$ of $G$ is a 
$p$-group for some prime $p$, and the subgroup $A=F\cap L$ is abelian.
\end{prop}

\begin{proof} 
We only need to consider the case when $F\not=1$. 
Let us show firstly, that the centralizer $C_F(L)$ is trivial.

To this end, assume that $D=C_F(L)\neq1$. 
Let $K$ denote the normalizer of the
block system $\mathcal B$ consisting of the $D$-orbits in $\Omega$.
By Lemma \ref{lem_1}, $L\leq K$.
Consider a $D$-orbit $\Delta\in \mathcal{B}$. 
By \cite[Theorem 4.2A]{dixon}, the centralizer $C$ of $D/C_D(\Delta)$ 
in $\mathrm{Sym}(\Delta)$ is isomorphic to a section of $D/C_D(\Delta)$,
hence nilpotent. Because the actions of $D$ and 
$L$ on $\Delta$ commute, the group $L/C_L(\Delta)$ is isomorphic
to a subgroup of $C$, hence nilpotent. On the other hand, $L$ is perfect.  
It follows that $L=C_L(\Delta)$.
Since this happens for every orbit $\Delta$, we obtain $L=1$.
But this is impossible, because $\lambda(G)=n\ge1$. 
Hence, $D=1$.
It follows immediately, that the intersection $A=L\cap F$ is non-trivial.

We will show now, that $A$ is a $p$-group.
To this end, assume that the order of $A$ is divisible by two different 
primes $p$ and $q$. 
Recall, that  $A$  is a normal subgroup in $G$, because $F$ and $L$ 
are normal subgroups in $G$. Moreover, $A$ is nilpotent.
Let $P$ resp.\ $Q$ be the Sylow $p$- resp.\ the Sylow $q$-subgroup of $A$.
Then $P$ and $Q$ are normal subgroups of $G$.
Since the stabilizer
$G_\omega$ of a point $\omega\in\Omega$ has trivial core, $P$ and $Q$
cannot be contained in $G_\omega$.
Application of Lemma \ref{lem_1} to
$QG_\omega$ in the role of $H$ yields $L\le QG_\omega$, and 
Dedekind's modular law gives $A=A\cap QG_\omega=Q(G_\omega\cap A)$.
It follows that $P\le G_\omega\cap A$, 
a contradiction to $P\not\le G_\omega$.

Now $A$ is a $p$-group for some prime $p$.
Assume next, that $F$ is not a $p$-group.
Then $F$ has a non-trivial $q$-subgroup $Q$ for some prime $q\neq p$.
Again, $Q$ is normal in $G$ and $QG_\omega>G_\omega$. Therefore,
$L\le QG_\omega$ and 
$A=F\cap L\le F\cap QG_\omega=Q(F\cap G_\omega)$.
It follows that $A\le F\cap G_\omega$, 
a contradiction to $A\not\le G_\omega$.

It remains to show, that $A$ is abelian.
Consider a group $B$, which is normalized by $G_\omega$ and satisfies
$F_\omega<B\le F$. Lemma \ref{lem_1} shows, that $L\le BG_\omega$, whence
$$ 
A=L\cap BG_\omega\cap F=L\cap B(G_\omega\cap F)=
L\cap BF_\omega=L\cap B\le B. 
$$
Assume, that there exist two distinct \emph{minimal} such groups $B_1$ and
$B_2$. This would imply
$
A\le B_1\cap B_2 = F_\omega\leq G_\omega,
$
in contradiction to $G_\omega$ having trivial core in $G$.
We conclude, that for each $\omega\in \Omega$, there is a 
unique minimal subgroup $B(\omega)$, which is normalized by $G_\omega$ 
and satisfies $F_\omega <B(\omega)\leq F$.

Because $F$ is nilpotent, the group $K(\omega)=N_F(F_\omega)$ 
is strictly larger than $F_\omega$. Moreover, $K(\omega)$ 
is normalized by $G_\omega$. Thus, $B(\omega)$ is contained in $K(\omega)$.
It follows, that the subgroup $A$ of $B(\omega)$ normalizes $F_\omega$ too.
Now $AF_\omega$ is a subgroup of $B(\omega)$, 
which is normalized by $G_\omega$. Since $A$ is not contained in $F_\omega$,
we obtain $AF_\omega=B(\omega)$. 

The normal subgroup $V=\bigcap_{\omega\in\Omega}B(\omega)$ of $G$ 
contains $A$. And so it suffices to show, that $V$ is abelian.
Consider a maximal subgroup $M$ of $B(\omega)$ containing $F_\omega$.
Because $F_\omega$ is normalized by $G_ \omega$, the $G_\omega$-core
$M_0=\bigcap_{g\in G_\omega}M^g$ of $M$ contains $F_\omega$.
By minimality of $B(\omega)$, we have $M_0=F_\omega$.
Since $F$ is nilpotent, the commutator subgroup $\big(B(\omega)\big)'$ of 
$B(\omega)$ is contained in the $G_\omega$-conjugates of $M$
and hence in $M_0=F_\omega$.
It follows that $V'
%=\left(\bigcap_{\omega\in\Omega}B(\omega)\right)'
\le \bigcap_{\omega\in\Omega}\big(B(\omega)\big)'\le
\bigcap_{\omega\in\Omega}F_\omega=1$\,.
\end{proof}

We shall show now that the Frattini subgroup $F$ is even trivial. 

\begin{prop}\label{trivial frattini} 
The group $G$ has trivial Frattini subgroup.
\end{prop}

\begin{proof}
Assume that $F\not=1$.
In the sequel, $p$ will denote the prime dividing the order of 
$F$ (Proposition \ref{frattini}).
We prove firstly, that $L$ splits over 
$A=L\cap F$ and that $S_1(G)$ splits over $F$.

Let $B=\Phi(L)$, assume $B\not=1$,  and consider any $B$-orbit $\Gamma$ 
in $\Omega$. 
If $\Gamma=\Omega$ would hold, we could apply the Frattini argument to 
obtain $G=BG_\omega$ for every $\omega\in\Omega$. 
But this is clearly impossible, because $B\leq F$ and $F$ consists of 
non-generators for $G$ (\cite[5.2.12]{robinson}). 
It follows, that the $B$-orbits form a proper block system in the
$G$-set $\Omega$.
By Lemma \ref{lem_1}, $L$ normalizes every $B$-orbit $\Gamma$.
In particular, the Frattini argument yields $L=BL_\gamma$ for 
any point $\gamma\in \Gamma$, and this is again impossible. 
This contradiction shows, that $B=1$.

Applying Proposition \ref{frattini} and \cite[Theorem 5.2.13]{robinson}, 
we obtain that $L$ splits over the abelian normal subgroup
$A$. A complement $K$ to $A$ in $L$ is a complement to $F$ in $S_1(G)$, 
so that $S_1(G)$ splits over $F$ too. 
Because $S_1(G)/F$ is a minimal normal subgroup in $G/F$, there exists a
simple group $S$ in $\mathcal{L}$ such that
$K=\prod_{i=1}^\ell S_i$
where each $S_i$ is a copy of $S$.
In the sequel, we will distinguish the two cases when $p$ is odd and when 
$p=2$, but try to treat them simultaneously.

When $\mathbf{p>2}$, let $R$ be any Sylow $2$-subgroup of $K$. 
Clearly $R=\prod_{i=1}^\ell R_i$, where each $R_i$ is a 
Sylow $2$-subgroup of $S_i$. When $\mathbf{p=2}$, Lemma \ref{dihedral}
gives us an Aut$(S)$-invariant $S$-conjugacy class 
of dihedral subgroups of order $2r$ in $S$,
for some odd prime $r$.
Let $\mathcal D_i$ denote the copy of this 
conjugacy class in $S_i$. We can now choose a 
subgroup $U_i\in \mathcal{D}_i$ for each $i\in\{1, \dots , \ell\}$ and
define $U=\prod_{i=1}^\ell U_i$ and $R=\prod_{i=1}^\ell R_i$ where $R_i$
is the Sylow $r$-subgroup in $U_i$.\vskip1ex

{\bf Next step.} \ Our next aim is  to show, that if $R\leq G_\omega$, 
then $K\leq G_\omega$.\vskip1ex

Assume by way of contradiction, 
that $R\leq G_\omega$ and $S_i\not\le G_\omega$ for some $i$.
Without loss, $i=1$. Because the simple group $S_1$ is generated by the 
conjugates of $R_1$ in $S_1$, there exists $s_1\in S_1$ such that 
${R_1}^{s_1}\not\le G_\omega$.
Let $Q=R^{s_1}$ and $Q_i=R_i^{s_1}$ for all $i$.
Clearly $Q_i=R_i$ when $i\not=1$. Therefore, $Q_i\leq G_\omega$ if and 
only if $i\neq 1$.
When $\mathbf{p=2}$, we also let $D=U^{s_1}$ and $D_i={U_i}^{s_1}$ for all $i$. 

In the case when $\mathbf{p>2}$, an application of the Frattini argument 
within the group $G/F$ yields  
$G/F=(S_1(G)/F)(Y/F)$, where $Y/F$ denotes the normalizer of the 
Sylow $2$-subgroup $QF/F$ in $G/F$.
For every $y\in Y$, the Sylow 2-subgroup $Q^y$ of $QF$ is conjugate to
$Q$ by a suitable element $f\in F$, and $yf^{-1}\in N_G(Q)$.
This shows that $G=S_1(G)X$ for $X=N_G(Q)$.

In the case when $\mathbf{p=2}$, the $G$-conjugates of $DF/F$  
are conjugate in $S_1(G)/F$. Therefore, the Frattini argument yields 
$G=S_1(G)N_G(DF)$. Because $Q$ is a Sylow $r$-subgroup of $DF$, it 
is possible to apply the Frattini argument again. It follows, that
$N_G(DF)=FDX$, where 
$X=N_{N_G(FD)}(Q)\geq D$. Altogether, $G=S_1(G) X$.\vskip1ex 

Conclusion:\quad $G=S_1(G)X$ \ where \
$\left\{\begin{array}{ll}
X=N_G(Q) & \text{for}~~p>2\\
X=N_G(Q)\cap N_G(FD) & \text{for}~~p=2
\end{array}\right.$\vskip1ex

Since $S_1(G)/F$ is minimal normal in $G/F$, the factorization 
$G=S_1(G)X$ entails, that $X$ acts transitively on the set
$\{S_iF \mid i=1, \dots, \ell\}$ by conjugation.
Consider $x\in X$ and indices $i,j$ such that $(S_iF)^x=S_jF$.
Then $Q_i^x \le Q\cap S_jF \le K\cap S_jF = S_j(K\cap F) = S_j$,
whence $Q_i^x \le Q\cap S_j = Q_j$. This shows, that
$X$ also acts tran\-sitively on the set
$\{Q_i\mid i=1, \dots, \ell\}$ by conjugation.
The group $X/(X\cap S_1(G))$ is $(n-1)$-rarefied, because 
it is isomorphic to $G/S_1(G)$. Moreover, $X<G$, because $Q$ is
not normal in $G$. Altogether, $\lambda(X)=n-1$. 

Consider $X_1=N_X(Q_1)$ as the stabilizer of $Q_1$ under the
conjugation action of $X$ on $\{Q_i\mid i=1, \dots, \ell\}$.
In the case when $\mathbf{p>2}$, we let $Z=X_1\cap X_\omega$. 
Clearly, $Q_1\not\le Z$ because 
$Q_1\not\le X_\omega$.
In the case when $\mathbf{p=2}$, we let $Z=(X_1 \cap X_\omega )(F\cap X)$.
In order to see, that  $Z$ does not contain $Q_1$ also in this case, we 
argue by contradiction:
Assume that $Q_1\leq Z$. 
Then 
$$
\omega Q_1\subseteq \omega Z= \omega (X_1 \cap X_\omega )(F\cap X)=
\omega (F\cap X).
$$
The group $X\cap F$ centralizes $Q$, because $[Q,X\cap F]\le Q\cap F=1$.
Therefore, 
$Q_1$ is normal in $Z$ and $\omega Q_1$ is a block under the action of $Z$ 
on $\omega Z$. 
It follows, that
$r = |\omega Q_1|$ divides  $|\omega (F\cap X)|=
|(F\cap X): (F\cap X_\omega)|$. But the latter is a power of 2, a 
contradiction.\vskip1ex 

We have now shown that $Q_1\not\le Z$ in all cases.\vskip1ex

However, $Z$ contains the normal subgroup 
$\widehat{Q_1}=\prod_{i=2}^\ell Q_i$ of $X_1$.
Hence we can choose an 
irreducible constituent $W$ of the $\mathbb CX_1$-module
$\mathrm{Ind}{\stackrel{X_1}{_Z}}(\mathbb{C})$, on which $Q_1$ acts 
non-trivially. 
Note that $\widehat{Q_1}$ acts trivially on $W$,  
because it acts trivially on the whole 
$\mathrm{Ind}{\stackrel{X_1}{_Z}}(\mathbb{C})$.
Since the normal subgroup $Q_1$ of $X_1$ acts non-trivially on $W$,
its fixed point space $C_W(Q_1)$ is a proper $X_1$-submodule of $W$,
whence $C_W(Q_1)=0$.
We choose a right transversal $T=\{t_1, t_2, \ldots, t_\ell\}$ 
of $X_1$ in $X$ such that $t_1=1$ and ${Q_1}^{t_i}=Q_i$ for all $i$.
Then 
$
V=\,\mathrm{Ind}{\stackrel{X}{_{X_1}}}(W)\,=\,
\bigoplus_{i=1}^\ell (W\otimes_{X_1} t_i)\,.   
$
We apply Lemma \ref{induced module} with 
\,$X,\,X_1,\,Q,\,Q_1$ \,in the roles of \,$G,\,H,\,M^*=\langle M_0^G\rangle,\,M_0$ \,(resp.) and
obtain, that $V$ is an irreducible $\mathbb CX$-module.
The 
subspaces $W_i=W\otimes_{X_1} t_i$ are blocks under the action of $X$, and 
each factor $Q_i$ acts non-trivially only on $W_i$. Note that $C_W(Q_1)=0$ 
entails $C_{W_i}(Q_i)=0$ for all $i$.

Let $\Gamma=\omega X$ and consider the permutation module 
$\mathbb{C}\Gamma$. We apply Lemma \ref{l:module}
with 
\,$X,\,X_1,\,Z,\,Q,\,Q_1$ \,in the roles of 
\,$G,\,H,\,R,\,M^*,\,M_0$ \,(resp.).
Note that the hypotheses of the lemma are satisfied, because  
$(G_\omega\cap X_1)\widehat{Q_1}\le X_\omega\cap X_1\le Z$.
It follows, that $\mathbb C\Gamma$ contains $V$ as a direct summand.

Consider the image $v=\sum_{i=1}^\ell (v_i\otimes t_i)$ 
of $\omega$ under the projection of $\mathbb C\Gamma$
onto $V$. Then $v\neq 0$ by Lemma \ref{projec}.
If $b\in Q_j$ for some $j\geq 2$, then $vb=v$ because
$Q_j\leq G_\omega$. On the other hand, every $W_i$ $(i\neq j)$ is
fixed pointwise by $Q_j$, whence $(v_i\otimes t_i)b=v_i\otimes t_i$.
Thus,
$$ 
v~=~vb~=~\Big(\sum_{i=1}^\ell v_i\otimes t_i\Big)b~=~
\Big(\sum_{i\not=j} (v_i\otimes t_i)\Big) + (v_j  \otimes t_j)b~. 
$$
We conclude, that $ v_j  \otimes t_j= (v_j  \otimes t_j)b$ for 
all $b\in Q_j$ and that $v_j\otimes t_j\in C_{W_j}(Q_j)=0$.
Hence $v\in W_1$.

Consider the block system $\mathcal B=\{W_i\mid 1\le i\le \ell\}$ in 
the $\mathbb CX$-module $W$. For every $x\in N_X(\mathcal B)$ and every 
$j\in\{1,\ldots,\ell\}$, the group $Q_j^x$ fixes 
$\bigoplus_{i\neq j} W_i$ pointwise. It follows, that $Q_j^x=Q_j$ for 
every $j$.
In particular, $N_X(\mathcal B)$ normalizes every component of $S_1(G)/F$,
whence $N_X(\mathcal B)S_1(G)/S_1(G)$ is solvable.
But then,  $N_X(\mathcal B)\leq R_2(G)$ and 
$\lambda(X/N_X(\mathcal B))=\lambda(X)=n-1$.

Let $w\in F_\infty$ be any reduced $k$-variable word 
of length $n$. Again, we distinguish our two cases.\vskip1ex

{\bf Case} $\mathbf{1.\quad p>2.}$\vskip0ex
\noindent
With respect to its action on $\mathcal B$, the group 
$X/N_X(\mathcal B)$ satisfies Theorem C. Hence there exist a Sylow 
2-subgroup $P$ in $X$ and $\ov{x}\in P^k$ such that 
$W_1 w_i(\ov{x})\not=W_1 w_j(\ov{x})$, for all $0\leq i<j\leq n-1$.
An application of Lemma \ref{wreath property}
with 
\,$X,\,X_1,\,Q,\,Q$ \,in the roles of \,$G,\,H,\,M^*,\,A^*$ \,(resp.)
yields an element $u\in vX_1$ and a $k$-tuple $\ov{b}\in Q^k$
such that 
$u w_i(\ov{b}\ov{x})\not =u w_j(\ov{b}\ov{x})$ for all $0\leq i<j\leq n$.
Because the 2-group $Q$ is normal in $X$, we have $Q\le P$ and
$\overline{b}\overline{g}\in P^k$.
If $u=vg$, with $g\in X_1$, then $u$ is the projection of $\delta=\omega g$
 on $V$.
Hence $\delta w_i(\ov{b}\ov{x})\not =\delta w_j(\ov{b}\ov{x})$, 
for all $0\leq i<j\leq n$.
But then, $G$ would not be a counterexample to Theorem C.
This contradition shows, that $R\le G_\omega$ implies $K\le G_\omega$
in the case when $p$ is odd.\vskip1ex

{\bf Case} $\mathbf{2.\quad p=2.}$\vskip0ex
\noindent
Since $Q\trianglelefteq D$ and $X=N_G(Q)\cap N_G(DF)$, we have 
$[D,X]\leq DF\cap N_G(Q)=DF\cap X=D(F\cap X)$.
In particular, $D(F\cap X)\trianglelefteq X$.
Pick a Sylow 2-subgroup $J_0$ in $D$. Then $J =J_0(X\cap F)$ is a Sylow 2-subgroup
in $D(X\cap F)$, and the Frattini argument gives $X=D(X\cap F)Y$, where 
$Y=N_X(J)$.
But $D(X\cap F)\le S_1(G)$, and so $G=S_1(G)X=S_1(G)Y$.
   
Clearly, $J_0=\prod_{i=1}^\ell J_i$ for certain 
Sylow 2-subgroups $J_i$ of $D_i$ $(1\le i\le\ell)$.
Moreover, $Y$ must permute the subgroups $J_i(X\cap F)$ $(1\le i\le\ell)$
transitively via conjugation, as well as the subgroups $S_iF$ $(1\le i\le\ell)$. 
Therefore, it follows as above with $J_i(X\cap F)$ in place of $Q_i$, that 
$Y$ acts transitively on $\mathcal B$ and that $Y/N_Y(\mathcal B)$
has non-solvable length $n-1$.

In particular, $Y/N_Y(\mathcal B)$ satisfies Theorem C, and there 
exist a  Sylow $2$-subgroup $P$ of $Y$ and  $\ov{x}\in P^k$ such that 
$W_1 w_i(\ov{x})\not=W_1 w_j(\ov{x})$, for all $0\leq i<j\leq n-1$. 
Let $Y_1=N_Y(J_1)=N_Y\big(J_1(X\cap F)\big)$.
Note, that the intersection $X\cap F$ acts trivially on $V$, 
because it is a normal subgroup of $X$
with trivial action on $W$. 
Therefore we can apply Lemma \ref{wreath property}
with the quotients of the groups
\,$Y,\,Y_1,\,J,\,J$ \,modulo $(X\cap F)$
\,in the roles of \,$G,\,H,\,M^*,\,A^*$ \,(resp.).
This yields an element $u\in vY_1$ and a $k$-tuple $\ov{b}\in J^k$
such that 
$u w_i(\ov{b}\ov{x})\not =u w_j(\ov{b}\ov{x})$ for all $0\leq i<j\leq n$.

Because the 2-group $J$ is normal in $Y$, we have $J\le P$ and
$\overline{b}\overline{g}\in P^k$.
If $u=vg$, with $g\in Y_1$, then $u$ is the projection of $\delta=\omega g$
 on $V$.
Hence $\delta w_i(\ov{b}\ov{x})\not =\delta w_j(\ov{b}\ov{x})$, 
for all $0\leq i<j\leq n$.
But then, $G$ would not be a counterexample to Theorem C.
This contradition shows, that $R\le G_\omega$ implies $K\le G_\omega$
in the case when $p=2$ too.\vskip1ex

Conclusion:\quad $R\le G_\omega$ implies $K\le G_\omega$ for all $p$.
\vskip1ex

Since the normal subgroup $F$ of $G$ is not contained in $G_\omega$,
Lemma \ref{lem_1} gives $S_1(G)=FL\leq FG_{\omega}$ 
and $S_1(G)=F(S_1(G)\cap G_\omega)$
for every $\omega\in\Omega$. 
Since $R$ is a 2-group resp.\ an $r$-group, a conjugate of $R$
is contained in a Sylow subgroup of $S_1(G)\cap G_{\omega}$.
Therefore $R\le G_\alpha$ for some $\alpha\in\Omega$, and hence 
$K\le G_\alpha$ as well. 

Clearly $R$ is normal in $N=N_G(R)$, and so $R\le G_\beta$ and
$K\le G_\beta$ for every $\beta\in\alpha N$. In particular, $K$
is contained in the intersection 
$$H=\bigcap_{\beta\in\alpha N} (S_1(G)\cap G_\beta).$$
In particular, $1\neq K\le H < S_1(G)$ and $H$ is normalized by $N$.
Consider $M=HN$.

An application of the Frattini argument to the subgroup $R$ of
$S_1(G)$ shows, that $G=S_1(G)N$. 
Since $S_1(G)=FK$ and 
$K\leq M$, we have that $G=FM$. 
However, $F$ consists of non-generators for the group $G$.
Therefore $G=M$. 
But this implies $\Omega=\alpha HN=\alpha N$
and $1\neq H\le\bigcap_{\beta\in\alpha N} 
G_\beta=\bigcap_{\omega\in\Omega} G_\omega=1$, a contradiction.
\end{proof}

%%%%%%%%%%%%%%%%%%%%%%%%%%  Section 5  %%%%%%%%%%%%%%%%%%%%%%%%

\section{Proof of the main theorems}
We are now well-prepared to take the final steps in the proof of Theorem C.

\begin{namedthm}{Theorem C}%\label{thm:C}
Let $G$ be a finite group with $\lambda(G)= n$, and let $\Omega$ be a
faithful transitive $G$-set. Then, for every $\omega\in\Omega$ and for every 
non-trivial reduced word 
$w=w(x_1,\ldots,x_k)\in F_\infty$ of length $n$,
there exist a Sylow $2$-subgroup $P$ of $G$ and a tuple $\overline{g}\in P^k$ 
such that the points \
$
\omega w_0(\overline{g}),\,\omega w_1(\overline{g}),\, 
\dots ,\, \omega w_n(\overline{g})
$ \
are pairwise distinct. In particular, $G\in\mathcal P_n$. 
\end{namedthm}

\begin{proof}
At first we shall explain, why the last claim follows from the central statement in Theorem C.
Consider a finite group $G$ with nonsolvable length $n$, and let $\Omega$ be a faithful $G$-set. 
Then there exist finitely many $G$-orbits $\Omega_1, \dots, \Omega_r$ in $\Omega$ such that
$N_1\cap\ldots\cap N_r = 1$, where $N_i$ denotes the kernel of the action of $G$ on $\Omega_i$.
From \cite[Lemma 2.5]{fuma}, there exists some $k$ such that $\lambda(G)=\lambda(G/N_k)$.
Application of the first part of Theorem C to the action of $G$ on $\Omega_k$ now gives
$G/N_k\in\mathcal P_n(\Omega_k)$, whence $G\in P_n(\Omega)$.
\vskip1ex
We shall now prove the central statement of Theorem C. To this end, assume that Theorem C 
is wrong, and consider a minimal counterexample $(G,\Omega)$ with $\lambda(G)=n$.
Recall, that minimality of the counterexample refers to the number $n+|G|+|\Omega|$,
and that $G$ acts transitively and faithfully on $\Omega$.
Moreover, since $G$ is a counterexample, there exist a point $\omega\in\Omega$ and a word
$w=w(x_1,\ldots,x_k)\in F_\infty$ of length $n$ such that the points
$\omega w_0(\ov g),\ldots,\omega w_n(\ov g)$ are not pairwise distinct for any choice 
of a tuple $\ov g\in P^k$ with entries from any Sylow 2-subgroup $P$ of $G$.
The point $\omega$ and the word $w$ will be kept fixed throughout the proof.

By Proposition \ref{trivial frattini}, the unique minimal normal 
subgroup $S_1(G)$ in $G$ is the product $\prod_{j\in \Delta} S_j$
of copies $S_j$ of a simple group  $S\in\mathcal{L}$.
Let $\Delta=\{1,\dots,\ell\}$.
We also write $S^*$ instead of $S_1(G)$ and let 
$\pi_j\colon S^*\longrightarrow S_j$ denote the canonical projection.
Lemma \ref{dihedral}
gives us an Aut$(S)$-invariant $S$-conjugacy class of dihedral subgroups of 
order $2p$ in $S$,
for some odd prime $p$.
Let $\mathcal D_j$ denote the copy of this 
conjugacy class in $S_j$. 
\vskip1ex  

We aim to show firstly, that $G=S^*N_G(D^*)$ for a certain 
subgroup  $D^*$ of $S^*$, which is the direct product of dihedral 
groups $D_j\in\mathcal D_j$ $(1\le j\le\ell)$.\vskip1ex

{\bf Case 1.} {The subgroup $G_\omega\cap S^*$ is not a subdirect product 
of the components $S_j$.}
\vskip0ex
\noindent
Without loss we assume that $(G_\omega\cap S^*)^{\pi_1}<S_1$. 
Then it is possible to find $D_1\in\mathcal D_1$
such that the Sylow $p$-subgroup $C_1$ of $D_1$ is not contained 
in $(G_\omega\cap S^*)^{\pi_1}$. Because $D_1$ is generated by 
its involutions, there is also a Sylow $2$-subgroup $Q_1\leq D_1$, 
which is not contained in $(G_\omega\cap S^*)^{\pi_1}$. For every
$j\not=1$ we choose $D_j\in \mathcal{D}_j$ arbitrarily and form 
$D^*=\prod_{j\in \Delta}D_j$.
The $G$-conjugates of $D^*$ are precisely the $S^*$-conjugates of $D^*$.
Therefore the Frattini argument yields $G=S^*N_G(D^*)$.
\vskip1ex

{\bf Case 2.} {The subgroup $G_\omega\cap S^*$ is  a subdirect product 
of the components $S_j$.}
\vskip0ex
\noindent
In this case, there exists a partition $\mathcal U=\{\Delta_i\mid i=1, \dots, m\}$
of $\Delta$ such that
$ G_\omega\cap S^*=\prod_{i=1}^m U_i $ where
each $U_i$ is a diagonal subgroup of 
$S(\Delta_i)=\prod_{j\in \Delta_i}S_j$.
Here, some of the $\Delta_i$ must contain at least two points.  
Without loss we may assume, that $|\Delta_1|\geq 2$
for the set $\Delta_1$ containing 1.
The conjugation action of $G_\omega$ on $G_\omega\cap S^*$ gives rise
to an action of $G_\omega$ on the set $\mathcal{U}$, and $G_\omega $ 
normalizes   
$\mathcal{O}=\{\Delta_i\in \mathcal{U}\mid \left|\Delta_i\right|\geq 2\}$.

Assume, that the action on $\mathcal{O}$ is not transitive, and consider
a $G_\omega$-orbit $\mathcal{I}\subseteq \mathcal{O}$.
Because we are in case 2, it is possible to choose $\mathcal I$ such that
the $G_\omega$-invariant subgroup 
$S(\mathcal{I})=\prod_{\Delta_i\in \mathcal{I}}\prod_{j\in \Delta_i}S_j$
is not contained in $G_\omega$.
From Lemma \ref{lem_1}, the monolith 
$S^*$ is contained in $S(\mathcal{I})G_\omega$, and Dedekind's identity
yields $S^*=S^*\cap S(\mathcal{I})G_\omega=
(S^*\cap G_\omega)S(\mathcal{I})$.
Since $\mathcal I\neq\mathcal O$, 
there exists $\ov\Delta\in\mathcal O\setminus\mathcal I$.
But now, the projection of 
$S^*=(S^*\cap G_\omega)S(\mathcal{I})$ on $S(\ov{\Delta})$ 
is isomorphic to $S$, a contradiction to $|\ov\Delta|\ge2$.
We conclude, that all the $\Delta_i$ of size $\ge2$ are permuted 
transitively by $G_\omega$. In particular, they all have the same size $d$. 

There are automorphisms $\sigma_2, \dots, 
\sigma_d\in \mathrm{Aut}(S_1)$ such that
$$U_1=\{(s,s^{\sigma_2}, \dots , s^{\sigma_d})\mid s\in S_1\}\,.$$
For each $j\in\Delta_1$, we can choose a subgroup $D_j\in \mathcal D_j$,
such that $U_1\cap\big(\prod_{j\in \Delta_1} D_j\big) =1$\,. \
We can proceed in the same fashion with every $\Delta_i\in\mathcal O$.
When $S_i\le G_\omega$, we choose $D_i\in\mathcal D_i$ freely.
Let $D^*=\prod_{i\in \Delta}D_i$.
Note, that $D^*\cap G_\omega$ is the product of precisely 
those factors $D_j$,
for which $S_j\leq G_\omega$.
In particular, $(D^*\cap G_\omega)^{\pi_1}=1$.
And also in this case, the Frattini argument gives $G=S^*N_G(D^*)$. 
\vskip1ex
We proceed to work with the group $D^*$ constructed in the two cases.
For each $j\in \Delta$ choose a Sylow $2$-subgroup $Q_j$ of $D_j$ 
(in case 1, the group $Q_1$ has already been choosen), and let $C_j$ be the 
Sylow $p$-subgroup of $D_j$.
As usual, we let $C^*=\prod_{j\in \Delta}C_j$ and 
$Q^*=\prod_{j\in \Delta}Q_j$.
Further applications of the Frattini argument yield
$G=S^*N_G(D^*)$ and $N_G(D^*)= D^*N_{N_G(D^*)}(Q^*)$, whence
$G=S^*X$ for $X=N_G(D^*)\cap N_G(Q^*)$.

It follows, that $X$ must act transitively via conjugation
on the set $\{D_j\mid j\in \Delta\}$. Moreover, $X<G$ and
$G/S^*\simeq X/(X\cap S^*)$ imply that $\lambda(X)=n-1$. 
Consider the subgroup $A=C^*X$ and note, that $A$ contains $D^*=C^*Q^*$.
We observe, that any non-trivial element $c$ from $C^*$ with non-trivial 
constituent in some $C_j$ can be conjugated by an involution in $D_j$,
 whence 
$C_j$ is contained in the normal closure of $c$ in $A$.
Together with the transitivity of the action of $X$ on the set $\{D_j\mid j\in \Delta\}$
this yields, that $C^*$ is a minimal normal subgroup in $A$.   
Now $C_1$ is not contained in $G_\omega$. It follows, that 
$C^*$ acts faithfully on the set $\Gamma=\omega A$.
Consider the kernel $K$ of the action of $A$ on $\Gamma$. 
Since $K\cap C^*=1$, the subgroup $K$ commutes with $C^*$.
But then $K$ normalizes every $C_j$ and hence every $S_j$.
It follows that $K\le R_2(G)$ and $\lambda(A/K)=\lambda(A)=n-1$. 

If $K$  contained an involution $a$ from $D^*$, then 
the subgroup $[a,D^*]$ of $K$ would contain one of the $C_j$.
This argument shows, that $D^*\cap K=1$. 
Consider $H=N_A(D_1)$ and
$R=(H\cap A_\omega)\widehat{D_1}$, where $\widehat{D_1}=\prod_{i\not=1}D_i$.
Assume, that $C_1\leq R$. 
Then $C_1\le D^*\cap R \le (D^*\cap G_\omega)\widehat D_1$. 
But this implies $1\neq C_1\le (D^*\cap G_\omega)^{\pi_1}$, 
a contradiction. 
It follows, that $C_1\not\le R$.
In particular, 
there exists then an irreducible component
$W$ in the $\mathbb CH$-module $\mathrm{Ind}{\stackrel{H}{_R}}(\mathbb{C})$ 
on which $C_1$ acts non-trivially. 

From Lemmata 
\ref{induced module} and \ref{l:module}, the $\mathbb CA$-module 
$V=W\otimes_{H}\mathbb{C}A$ is (isomorphic to) a submodule 
of $\mathbb{C}\Gamma$.
Clearly, $V= \oplus_{t\in T}W_t$ with $W_t=W\otimes t$, where 
$T$ denotes a transversal of $H$ in $A$.
Since $C^*\le H$, the transversal can be chosen such that 
$1\in T\subset X$.
Note that $X$ normalizes $Q^*$.
Therefore, each index $j\in\Delta$ determines a unique 
$t\in T$ such that $D_j=D_1^t$, and even $C_i=C_1^t$ and $Q_i=Q_1^t$.  
Moreover, $D_j$ acts non-trivially on $W_t$ if and only if $D_j=D_1^t$,
because $\widehat D_1$ acts trivially on $W$.

Let $v=\sum_{t\in T}v_t\otimes t$ denote the projection of a point 
$\tau\in\Gamma$ on $V$. Then $v\neq 0$ by Lemma \ref{projec}, and we may choose 
$\tau$ in such a way, that $v_1\neq0$. 
Consider any reduced $k$-variable word $w\in F_\infty$ of length $n$.
The group $G$ permutes the set $\{S_j\mid j\in\Delta\}$ of components of
$S^*$.
Because $G$ is a minimal counterexample,
there exist a Sylow $2$-subgroup $P_0$ of $X$ and a $k$-tuple
$\overline{g}\in P_0^k$ such that the components
$ {S_1}^{w_0(\ov{g})}, \, {S_1}^{w_1(\ov{g})}, \, 
\dots, {S_1}^{w_{n-1}(\ov{g})}$ 
are pairwise distinct.
In particular, the subgroups
${D_1}^{w_0(\ov{g})}, \, {D_1}^{w_1(\ov{g})}, \, 
\dots, {D_1}^{w_{n-1}(\ov{g})}$ are pairwise distinct as well.

Clearly, $W=\langle v_1H\rangle$. And since $D_1=\langle Q_1^{C_1}\rangle$ 
acts non-trivially on $W$, the group $Q_1$ acts non-trivially on $W$ too. 
We can therefore apply Lemma \ref{wreath property} 
with \,$A,\,H,\,D_1,\,Q_1$
in place of \,$G,\,H,\,M_0,\,A_0$ \,(resp.).
It follows, that there are a $k$-tuple $\ov{b}\in {Q^*}^k$ 
and a vector $u_1=v_1h\in  v_1H$, 
such that the elements $(u_1\otimes 1)w_j(\ov b\ov g)$ are 
pairwise distinct for $0\le j\le n$.
Here, $\overline{b}\overline{g}\in (Q^*P_0)^k$ and $Q^*P_0$ 
is contained in a Sylow $2$-subgroup $P$ of $G$.

The projection \,$u =vh=\sum_{t\in T}u_t\otimes t$ 
\,of $\tau h$ on $V$ has non-trivial component 
$u_1\otimes 1$ along $W_1$. 
We now refine our choice of $T$ by asking that $T$ contains a 
transversal of $Z=H\cap P$ in $P$. 
Consider the elements $x_{ij}=w_i(\ov{b}\ov{g}){w_j(\ov{b}\ov{g})}^{-1}$
for $0\leq i<j\leq n$. 
Let $x$ be one of the elements $x_{ij}\in P$. 
Then $u_1\otimes 1\not=(u_1\otimes 1)x=u_1z\otimes t$,
where $x=zt$ for unique $z\in Z$ and $t\in T\cap P$.
Since $Z$ is a 2-group, $|u_1Z|$ is a power of 2.

Assume that $u_tC_1\subseteq u_1Z$. 
The set $u_tC_1$ is a block for the action of $H$ on $u_tH$ 
and therefore it is also a block for the action of $Z$ on $u_tZ=u_1Z$.
In particular, also $|C_1:\mathrm{Stab}_{C_1}(u_t)|=|u_tC_1|$ 
is a power of 2. 
On the other hand, the $p$-group $C_1$ is a normal subgroup of $H$,
which acts non-trivially on $W$, so that it must in fact act
fixed-point-freely on $W$, with orbits of size $p$.
This contradiction shows, that $u_tC_1\not\subseteq u_1Z$.

For each $t\in T\setminus\{1\}$, 
we choose an element $s_t\in C_1$ such that
$u_ts_t\notin u_1Z$. And we let $s_1=1$.
Consider $s^*=\prod_{t\in T}s_t^t\in\prod_{t\in T}C_1^t=C^*$ and
the point $\alpha=\tau hs^*\in\Gamma$.
The projection of $\alpha$ on $V$ is
$$
\tilde v\,=\,us^*\,=\,u_1\otimes 1\,+\sum_{1\not=t\in T}u_ts_t\otimes t.
$$
We claim that $\tilde vx_{ij}\not=\tilde v$ for $1\le i< j\le n$.

To this end we only need to show, that for every choice of $i<j$ 
there exists some $t\in T$ such that the $W_t$-components of $\tilde vx_{ij}$ 
and $\tilde v$ are different.
In the case when $x_{ij}$ normalizes $W_1$, the $W_1$-component of 
$\tilde vx_{ij}$ is $(u_1\otimes 1)x_{ij}\neq u_1\otimes 1$.
In the case when $W_1x_{ij}=W_t\neq W_1$, the $W_t$-component 
of $\tilde vx_{ij}$ has the form $y\otimes t$ for some $t\in u_1Z$, while
the $W_t$-component of $\tilde v$ is $(u_ts_t)\otimes t\not= y\otimes t$.

We have shown now, that $\alpha x_{ij}\not= \alpha$ resp.\
$\alpha w_i(\ov{b}\ov{g})\not= \alpha{w_j(\ov{b}\ov{g})}$ for all 
$0\leq i<j\leq n$. 
Recall that $\Gamma=\omega A$, so that $\omega=\alpha a$ for some $a\in A$.
Now $\omega w_i((\ov{b}\ov{g})^a) = \alpha w_i(\ov{b}\ov{g})a$ for all $i$, and so
it follows, that the points
$\omega w_i((\ov{b}\ov{g})^a)$ are pairwise distinct for $0\le i\le n$.
However, also the entries of the tuple $(\ov{b}\ov{g})^a$ lie in a common Sylow 2-subgroup
of $G$. This contradicts the choice of $\omega$ and $w$ at the very beginning of the 
proof. 
\end{proof}

\emph{Proof of Theorem B.}
Consider any faithful transitive $G$-set $\Omega$. By Theorem C,  
the group $G$ satisfies $\mathcal{P}_{n}$ on $\Omega$. 
Thus we can find $\omega\in \Omega$ and
$\ov{g}\in G^k$ such that 
$\omega w(\ov{g})=\omega w_n(\ov{g})\not=\omega$.
This implies immediately, that $w(\ov{g})\not =1$.
\hfill $\Box$\vskip2ex

\emph{Proof of Theorem A.}
Consider a non-trivial $k$-variable reduced word $w\in F_\infty$ 
of length $n=\lambda(G)$.
Theorem B provides a $k$-tuple $\ov{g}\in H^k$ such that $w(\ov{g})\not=1$. 
Therefore, $w$ is not a law in $G$.
This shows, that $\lambda(G)<\nu(G)$.
\hfill $\Box$%\\[2ex]

%%%%%%%%%%%%%%%%%% Bibliography %%%%%%%%%%%%%%%%%%%%%%%%%%%%%%%%%%%%%

\bibliographystyle{plain}
\bibliography{Larsen}

\begin{thebibliography}{10}

\bibitem{batta}
Meenaxi Bhattacharjee.
\newblock The ubiquity of free subgroups in certain inverse limits of groups.
\newblock {\em J. Algebra}, 172(1):134--146, 1995.

\bibitem{bors}
A.~Bors and A.~Shalev.
\newblock Words, permutations, and the nonsolvable length of a finite group.
\newblock {\em arXiv:1904.02370v1}, 2019.

\bibitem{atlas}
J.~H. Conway, R.~T. Curtis, S.~P. Norton, R.~A. Parker, and R.~A. Wilson.
\newblock {\em Atlas of finite groups}.
\newblock Oxford University Press, Eynsham, 1985.
\newblock Maximal subgroups and ordinary characters for simple groups, With
  computational assistance from J. G. Thackray.

\bibitem{CurtisReiner}
Charles~W. Curtis and Irving Reiner.
\newblock {\em Methods of representation theory. {V}ol. {I}}.
\newblock Wiley Classics Library. John Wiley \& Sons, Inc., New York, 1990.
\newblock With applications to finite groups and orders, Reprint of the 1981
  original, A Wiley-Interscience Publication.

\bibitem{dixon}
John~D. Dixon and Brian Mortimer.
\newblock {\em Permutation groups}, volume 163 of {\em Graduate Texts in
  Mathematics}.
\newblock Springer-Verlag, New York, 1996.

\bibitem{fuma}
Francesco Fumagalli, Felix Leinen, and Orazio Puglisi.
\newblock A reduction theorem for nonsolvable finite groups.
\newblock {\em Israel J. Math.}, 232(1):231--260, 2019.

\bibitem{GLS3}
Daniel Gorenstein, Richard Lyons, and Ronald Solomon.
\newblock {\em The classification of the finite simple groups. {N}umber 3.
  {P}art {I}. {C}hapter {A}}, volume~40 of {\em Mathematical Surveys and
  Monographs}.
\newblock American Mathematical Society, Providence, RI, 1998.
\newblock Almost simple $K$-groups.

\bibitem{jones}
Gareth~A. Jones.
\newblock Varieties and simple groups.
\newblock {\em J. Austral. Math. Soc.}, 17:163--173, 1974.
\newblock Collection of articles dedicated to the memory of Hanna Neumann, VI.

\bibitem{khukhro}
Evgenii~I. Khukhro and Pavel Shumyatsky.
\newblock Nonsoluble and non-p-soluble length of finite groups.
\newblock {\em Israel J. Math.}, 207(2):507--525, 2015.

\bibitem{larsen1}
Michael Larsen and Aner Shalev.
\newblock Word maps and {W}aring type problems.
\newblock {\em J. Amer. Math. Soc.}, 22(2):437--466, 2009.

\bibitem{larsen2}
Michael Larsen, Aner Shalev, and Pham~Huu Tiep.
\newblock The {W}aring problem for finite simple groups.
\newblock {\em Ann. of Math. (2)}, 174(3):1885--1950, 2011.

\bibitem{leinen}
Felix Leinen and Orazio Puglisi.
\newblock Free subgroups of inverse limits of iterated wreath products of
  non-abelian finite simple groups in primitive actions.
\newblock {\em J. Group Theory}, 20(4):749--761, 2017.

\bibitem{ore}
Martin~W. Liebeck, E.~A. O'Brien, Aner Shalev, and Pham~Huu Tiep.
\newblock The {O}re conjecture.
\newblock {\em J. Eur. Math. Soc. (JEMS)}, 12(4):939--1008, 2010.

\bibitem{robinson}
Derek J.~S. Robinson.
\newblock {\em A course in the theory of groups}, volume~80 of {\em Graduate
  Texts in Mathematics}.
\newblock Springer-Verlag, New York, second edition, 1996.

\bibitem{sha}
Aner Shalev.
\newblock Some results and problems in the theory of word maps.
\newblock In {\em Erd\"{o}s centennial}, volume~25 of {\em Bolyai Soc. Math.
  Stud.}, pages 611--649. J\'{a}nos Bolyai Math. Soc., Budapest, 2013.

\end{thebibliography}
\end{document}